\documentclass[12pt]{article}

\usepackage{amsmath,amsfonts,amssymb,amsthm,amscd,url}

\title{Saturated free algebras revisited}

\date{\today}

\author{Anand Pillay\\University of Notre Dame
  \and Rizos Sklinos\\University Lyon 1 }

  \newtheorem{Theorem}{Theorem}[section]
\newtheorem{Proposition}[Theorem]{Proposition}
\newtheorem{Definition}[Theorem]{Definition} 
\newtheorem{Remark}[Theorem]{Remark}
\newtheorem{Lemma}[Theorem]{Lemma}
\newtheorem{Corollary}[Theorem]{Corollary}
\newtheorem{Fact}[Theorem]{Fact}
\newtheorem{Example}[Theorem]{Example}
\newtheorem{Question}[Theorem]{Question}

\newcommand{\Q}{\mathbb Q}  
\newcommand{\Z}{\mathbb Z}

\begin{document}
\maketitle

\begin{abstract} 
We give an exposition of results of Baldwin-Shelah \cite{Baldwin-Shelah} on saturated free algebras, 
at the level of generality of complete first order theories $T$ with a saturated model $M$ which is in the algebraic closure of an indiscernible set. 
We then make some new observations when $M$ {\em is} a saturated free algebra, analogous to (more difficult) results for the free group, such as a description of forking. 
\end{abstract}

\section{Introduction}
This paper has its origin in joint discussions during the second author's work on his  Ph.D. thesis in Leeds. 
Although the topic of the thesis was the model theory of the free noncommutative group, 
we were interesting in analogies with the  much easier situation of saturated free algebras, 
which  had been studied in a paper of Baldwin and Shelah \cite{Baldwin-Shelah}. (But note that free groups, although stable, are never saturated.)
In Section 2 we recapitulate, with quick proofs, the main results of \cite{Baldwin-Shelah}, in the  more general model-theoretic context described in the abstract, 
which was already alluded to in \cite{Baldwin-Shelah}. These results consist of $\omega$-stability and  finite-dimensionality of $T$, 
and some refinements involving decompositions of suitable models of $T$ as the algebraic closure of Morley sequences in  weight one types.  
In Section 3, we  look in more detail at a basis (or free generating set) $I$ of a saturated free algebra, proving various 
results which are more specific to the case at hand and not necessarily valid at the level of generality of section 2. 
For example we prove that $I$ is a Morley sequence in a stationary type over $M$, and we describe forking in $M$ in terms of free decompositions. 
We also ask several questions, some of which may have easy answers. 
In section 4 we give a few  examples, mainly highlighting the distinction between the context of Section 2 and that of Section 3.


Our model theory notation is standard. For simplicity we will work throughout with countable languages and theories.  
If $L$ is a language consisting only of function symbols then we will call an $L$-structure an $L$-algebra. 
In that case, by a {\em variety} $V$ in the language $L$ (in the sense of universal algebra) we mean 
a class of $L$-algebras axiomatized by a collection of so-called {\em identities}, namely universal closures of expressions 
$t_{1}(\bar x) = t_{2}(\bar x)$ where $t_{1}, t_{2}$ are $L$-terms. 
Free algebras exist in $V$: the free algebra $F_{X}$ on generators $X$ is characterized by the 
property that $X$ generates $F_{X}$ as an algebra and any map from $X$ to an algebra $A\in V$  extends to a homomorphism from $F_{X}$ to $A$, necessarily unique. 
It is clear that any permutation of $X$ extends to an automorphism of $F_{X}$ whereby $X$ will be an {\em indiscernible set} in $F_{X}$ 
which of course generates $F_X$ under the terms of $L$.

Up to and including the 1970's there was considerable interaction between universal algebra and model theory, 
and it was natural for Baldwin and Shelah to study algebras $F$ which are both free (on some set of generators) and uncountably saturated (as first order structures). 
In the paper \cite{Baldwin-Shelah} a number of interesting structural results are proved about $Th(F)$, although the status of Theorem 2 there is unclear. 
As they already mention in their paper, all their structural results hold with appropriate modifications under the {\em weaker} 
assumption that $T$ is what we call below {\em almost indiscernible}, namely has a model $M$ which is 
uncountably saturated and in the algebraic closure of some indiscernible set. 
In any case, working in this slightly more general context of almost indiscernible theories, 
we give a quick account of the main lemmas of \cite{Baldwin-Shelah} and a correct but slightly weaker version of their Theorem 2. 

One would imagine on the other hand that there are model-theoretic or stability theoretic properties of saturated free algebras which 
are more specific and do not generalize to almost indiscernible theories, and sections 3 and 4 explore this topic.


Note that in the informal definition above of an almost indiscernible theory,  we say {\em indiscernible set}  not {\em indiscernible  sequence}. 
If we say rather {\em sequence} then this is a weaker notion which could be explored separately. 
In fact in \cite{Mariou} Benoist Mariou studies countable first order theories $T$ with a saturated model 
$M$ which has an expansion $M'$ in a countable language such that $M'$  itself is in the algebraic closure of an indiscernible sequence. 
Mariou proves that such a theory is $NIP$ (does not have the independence property) and moreover among stable theories the property 
characterizes the $\omega$-stable theories.  This whole topic is of course related to the old proof of $\omega$-stability of uncountably categorical theories, 
using Ehrenfeucht-Mostowski models.  

\vspace{2mm}
\noindent

In any case if $T$ is the theory of a free (uncountable) saturated algebra then $T$ is almost indiscernible and really 
the results in \cite{Baldwin-Shelah} were about such theories.  

So in section 2 we will study almost indiscernible theories and make reasonably free use of stability theoretic notions. 
Saharon Shelah invented stability theory,  and the fundamental notions of the subject trace back to him, although alternative expositions, proofs, 
and even definitions, have been developed  by others. At the time of the writing of \cite{Baldwin-Shelah}, Shelah's  \cite{Shelah} 
was the only source in book form for stability theory, although some other important and influential  papers were in circulation, 
such as \cite{Lascar-Poizat} and \cite{Makkai}. In the meantime several books on stability theory have appeared, and a common 
vocabulary and conceptual framework has been more or less established among practitioners of the subject. 
Chapter 1 of \cite{Pillay-GST} is devoted to a summary, with selected proofs, of stability theory, and we will in particular  use section 4 
(Miscellaneous  facts about stable theories), as a basic reference for the current paper. The reader might also wish 
to consult \cite{Baldwin}, \cite{Lascar-book} as well as the paper \cite{Pillay-nmd} which gives 
an exposition of the computation of the spectrum function for $\omega$-stable nonmultidimensional theories. 
In any case we complete this introduction with some facts about nonmultidimensional $\omega$-stable theories. 

{\bf We fix a complete $\omega$-stable theory $T$ and work in a big saturated model  (${\bar M}$ say).}
\newline
\noindent
 A {\em complete type} usually  refers to a complete type over some small subset of  ${\bar M}$. Sometimes we refer to global types which are complete types over ${\bar M}$.  
 Regular types are assumed, among other things, to be stationary.

\begin{Definition} 
Recall that $T$ is said to be nonmultidimensional (finite-dimensional) if there are only boundedly many (finitely many) regular types, up to nonorthogonality. 
\end{Definition} 

In fact for a general stable theory nonmultidimensionality is defined as any two nonalgebraic stationary types being nonorthogonal, 
which is equivalent to the definition above for superstable $T$. 

In \cite{Baldwin-Shelah} there are statements to the effect that arbitrary models of an $\omega$-stable finite-dimensional theory are 
prime over a finite union of indiscernible sequences related to regular types. We want to make this a little more precise (and correct).  
Let  $M_{0}$ denotes a copy of the prime model as an elementary substructure of ${\bar M}$.

\begin{Fact} Suppose $T$ is nonmultidimensional. Then up to nonorthogonality every regular type can be chosen 
as a type over $M_{0}$ (which moreover is strongly regular). 
\end{Fact}

{\bf We now assume in addition that $T$ is nonmultidimensional. }\
\newline
\noindent
Let now $p_{i}(x)$ for $i\in I$ be a list of regular types over $M_{0}$, up to nonorthogonality, and let $a_{i}$ 
be a finite tuple from $M_{0}$ such that $p_{i}$ is definable over $a_{i}$. 
Let $(p_{i})_{a_{i}}(x)$ be the restriction of $p_{i}(x)$ to $a_{i}$. Note that $|I|\leq \omega$ (as for example $S(M_{0})$ is countable). 

\begin{Fact} 
Let $M$ be any elementary substructure of ${\bar M}$. Assume that $M$ contains $a_{i}$ for each $i$. 
Let $J_{i}$ be a maximal independent set of realizations of $(p_{i})_{a_{i}}$ in $M$, for each $i$. 
Then $M$ is prime (and minimal) over $\bigcup_{i}a_{i} \cup\bigcup_{i\in I} J_{i}$.
\end{Fact} 

The cardinality of $J_{i}$ in $M$ depends only on $M$ and $(p_{i})_{a_{i}}$ and is written $dim((p_{i})_{a_{i}},M)$. 

\begin{Remark} 
The condition that $M$ contains all the $a_{i}$ is minor. 
In general $M$ contains an isomorphic (elementary) copy $M_{0}'$ of the prime model $M_{0}$, 
and so will contain $a_{i}'$ for $i\in I$ such that $tp((a_{i}':i\in I)) = tp((a_{i}:i\in I))$. 
Then work instead with the ``copies" of the $(p_{i})_{a_{i}}$'s over the $a_{i}'$. 
When it comes to counting models, it becomes important to note that if $a_{i}$ and $a_{i}'$ have the same strong type over $\emptyset$, 
then $(p_{i})_{a_{i}}$ and its copy over $a_{i}'$ have the same ``dimension" in an model containing both $a_{i}$ and $a_{i}'$. 
\end{Remark}

The notion of an ``$a$-model"  (see Definition 4.2.2 of Chapter 1 of \cite{Pillay-GST}) 
is important and for $\omega$-stable $T$ coincices with an $\omega$-saturated model.

\begin{Lemma} 
Let $M$ be as in Fact 1.3. Then $M$ is $\omega$-saturated iff $|J_{i}|$ is infinite for each $i\in I$.
\end{Lemma} 
\begin{proof} 
Note that there is a unique countable $\omega$-saturated model of $T$ which we will call $M_{\omega}$ and which we can assume to contain $M_{0}$. 
Now suppose $M$ is such that such that for each $i$, $dim((p_{i})_{a_{i}},M)$ is infinite. Let $b$ be a finite tuple from $M$ and $r(y,b)$ a complete type over $b$. 
Then there are countably infinite $J_{i}'\subset J_{i}$ for $i\in I$ such that $b$ is contained in an elementary 
submodel $M'$ of $M$ prime over the $a_{i}$'s together with the $J_{i}'$. So $M'$ is isomorphic to $M_{\omega}$ hence $\omega$-saturated too. 
So $r(x,b)$ is realized in $M'$ so in $M$. Hence $M$ is $\omega$-saturated. 
\end{proof}

\begin{Corollary}\label{1.6}
Any elementary extension of an $\omega$-saturated model of $T$ is also $\omega$-saturated. 
\end{Corollary} 

Now we can in fact choose $a_{i}$ to be the canonical base of $p_{i}$ (as an element of $M_{0}^{eq}$), 
and it is well-known that then $a_{i}$ is in the definable closure of $J$ whenever $J$ is an infinite Morley sequence in $(p_{i})_{a_{i}}$. 
We conclude:

\begin{Proposition}  
Any $\omega$-saturated  model $M$ of $T$ is prime over a union of indiscernible sets.
\end{Proposition}
\begin{proof} 
We may assume $M$ contains $M_{0}$ so $a_{i}$ for each $i\in I$. Let $J_{i}$ be a maximal Morley sequence in $M$ in $(p_{i})_{a_{i}}$. 
Then $J_{i}$ is infinite, whereby  $a_{i}\in dcl(J_{i})$. So by Fact 1.3, $M$ is already prime over the union of the $J_{i}$. 
\end{proof}

\begin{Remark} 
Of course when $T$ is finite-dimensional $I$ is finite. 
At the current level of generality, Theorem 1 of \cite{Baldwin-Shelah} seems only valid for $\omega$-saturated models of $T$. 
\end{Remark}

Finally note by Corollary \ref{1.6} that:

\begin{Remark} 
Let $M$ be an $\omega$-saturated model of $T$, and $B$ any set. 
Then the prime model over $M\cup B$ coincides with the $a$-prime model (prime model in the category of $\omega$-saturated models) over $M\cup B$. 
\end{Remark}

Both authors would like to thank John Baldwin for some useful correspondences. 
Also Baldwin acknowledged that Theorem 2 of \cite{Baldwin-Shelah} may need some fine-tuning to be correct, 
and quite possibly this will be done independently by Baldwin and Shelah. 
The second author would like to thank Artyom Chernikov for pointing out the connection with Mariou's work \cite{Mariou}.

\section{Almost indiscernible theories}

We work with a countable language $L$ and complete $L$-theory $T$.

\begin{Definition} 
$T$ is said to be almost indiscernible if there is a  saturated model of $T$ of cardinality $\aleph_{1}$ 
which is in the algebraic closure of an indiscernible {\em set} of finite tuples $I$ (so $I$ is forced to have cardinality $\aleph_{1}$ too). 
\end{Definition}

\vspace{2mm}
\noindent
{\bf Assumption.}  $T$ is almost indiscernible. 

\vspace{2mm}
\noindent
So we let $M$ denote a saturated model of $T$ of cardinality $\aleph_{1}$ which is in the algebraic closure of an 
indiscernible set (which we write as a sequence)  $I = (e_{\alpha}:\alpha<\aleph_{1})$ of cardinality $\aleph_{1}$.

\vspace{2mm}
\noindent
Let $\bar\kappa$ be a cardinal much bigger than $\aleph_{1}$. Let $\bar M$ be a $\bar\kappa$-saturated elementary extension of $M$. 
Let ${\bar I} = (e_{\alpha}: \alpha <\bar\kappa)$ be an indiscernible set in $\bar M$ extending $I$. 
For each infinite $\lambda \leq\bar \kappa$, let $I_{\lambda} = (e_{\alpha}:\alpha<\lambda)$ and let $M_{\lambda} = acl(I_{\lambda})$ inside ${\bar M}$. 
So $M_{\aleph_{1}} = M$ is an elementary substructure of $\bar M$ by definition of $\bar M$ but on the face of it the other $M_{\lambda}$'s are just subsets of $\bar M$. 
Note that $M_{\lambda}$ has cardinality $\lambda$. We then easily obtain:

\begin{Lemma}\label{saturation}
\ \begin{itemize}
 \item[(i)] The $M_{\lambda}$ form  an elementary chain.
 \item[(ii)] $M_{\omega}$ is $\omega$-saturated.
\end{itemize}
\end{Lemma} 
\begin{proof} 
(i) is left to the reader.
\newline
(ii) Let $\Sigma(x)$ be a partial type over a finite subset $A$ of $M_{\omega}$, consistent with ${\bar M}$. 
Then $A$ is contained in the algebraic  closure of $e_{1},\ldots,e_{n}$ say, and $\Sigma(x)$ is realized in $M_{\aleph_{1}}$ 
by some $d$ in the algebraic closure of $e_{1},\ldots,e_{n}$ together with some other $e_{\alpha_{1}},\ldots,e_{\alpha_{k}}$ with $\alpha_{i} < \aleph_{1}$. 
Then as $tp(e_{1},\ldots,e_{n},e_{\alpha_{1}},\ldots,e_{\alpha_{k}}) = tp(e_{1},\ldots,e_{n},e_{n+1},\ldots,e_{n+k})$ we can find such a realization in $M_{\omega}$. 
\end{proof}

\begin{Remark} 
In fact one can also show directly at this stage that each $M_{\lambda}$ is $\lambda$-saturated, 
although it will also follow easily from $\omega$-stability, proved next. 
\end{Remark}

\begin{Proposition} 
$T$ is $\omega$-stable.
\end{Proposition}
\begin{proof} By Lemma \ref{saturation}(ii), it suffices to show that there only countably many complete $1$-types over $M_{\omega}$. 
Now any such type is of the form $tp(d/M_{\omega})$ for some $d\in M_{\omega_{1}}$ and $d\in acl(M_{\omega}\cup I')$ 
where $I'= I_{\omega_{1}}\setminus I_{\omega}$. So $tp(d/M_{\omega}\cup I')$ is isolated by some formula $\phi(x,{\bar 
e})$ where $\phi(x,{\bar y})$ has parameters from $M_{\omega}$ and $\bar e$ is a finite tuple from $I'$. 
As the type over $M_{\omega}$ of such a finite tuple $\bar e$ is determined by cardinality of $\bar e$ 
we see really that $tp(d/M_{\omega})$ is determined by the formula $\phi(x,{\bar y})$  (which includes the length of $\bar y$). As there are countably many 
possibilities there are countably many such types. 
\end{proof}

Concerning saturation of the $M_{\lambda}'$: let $q(x)$ be a complete type over a subset $A$ of $M_{\lambda}$ of cardinality $<\lambda$. 
We may assume that $A$ contains $M_{\omega}$.

Let now $p$ denote the so-called average type of $I$ over ${\bar M}$. Namely  $p(x)\in S({\bar M})$ and for $\phi(x)$ over ${\bar M}$, $\phi(x)\in p$ 
if $\phi(e_{i})$ holds for all but finitely many $i<\kappa$.  By $\omega$-stability $p$ is definable over $\{e_{i}:i<n\}$ for some finite $n$, in particular 
definable over $M_{\omega}$ and moreover $p|M_{\omega} = tp(e_{\omega}/M_{\omega})$ and ${\bar I}$ is a Morley sequence in $p|M_{\omega}$. 

\begin{Lemma} 
Any complete type over $M_{\omega}$ is nonorthogonal to $p$, hence nonweakly orthogonal to $p$ as $M_{\omega}$ is an $a$-model.
\end{Lemma}
\begin{proof} 
Let $q(y)\in S(M_{\omega})$. Then as $M = M_{\aleph_{1}}$ is $\aleph_{1}$-saturated $q$ is realized by some $d\in acl(M_{\omega},{\bar  e})$ 
for some finite ${\bar e}$ from $I$. As ${\bar e}$ is an independent set of realizations of $p|M_{\omega}$ it follows that $q$ is nonorthogonal to $p$. 
\end{proof}

\begin{Proposition} 
$T$ is finite-dimensional.
\end{Proposition} 
\begin{proof} 
For any regular type $r$ over $M_{\omega}$, by the previous lemma there is a realization $a_{r}$ of $r$ such that $e_{\omega}$ forks with $a_{r}$ over $M_{\omega}$. 
If $r_{1},\ldots,r_{n}$ are pairwise orthogonal regular types then the $a_{r_{i}}$ are independent over $M_{\omega}$ 
and each forks with $e_{\omega}$ over $M_{\omega}$. So the weight of $p$ gives a bound on $n$. 
Hence there are at only finitely many regular types over $M_{\omega}$ up to nonorthogonality. 
By $\omega$-stability and the fact that $M_{\omega}$ is an $a$-model, this implies that $T$ is finite-dimensional. 
\end{proof}

So we see by Proposition 1.7 and its proof that any $\omega$-saturated model of $T$ is prime over a 
finite union of indiscernible sets each of which comes from a nonorthogonality class of a 
(strongly) regular type of $T$.  This is (suitably adapted) Theorem 1 of \cite{Baldwin-Shelah}.  We now make a few refinements.

\begin{Proposition} 
$M_{\omega + 1} = acl(M_{\omega}, e_{\omega})$ is prime and $a$-prime (and minimal) over $(M_{\omega},c_{1},\ldots,c_{n})$ where 
$tp(c_{i}/M_{\omega})$ is regular, $\{c_{1},\ldots,c_{n}\}$ is $M_{\omega}$-independent, and each regular $q\in S(M_{\omega})$ 
appears up to nonorthogonality among the $tp(c_{i}/M_{\omega})$. 
\end{Proposition} 
\begin{proof} 
Now $M_{\omega + 1}$ is clearly prime over $(M_{\omega},e_{\omega})$ and by Remark 1.10 is also $a$-prime over 
$(M_{\omega},e_{\omega})$. Let $\{c_{1},\ldots,c_{n}\}$ be a maximal independent over $M_{\omega}$ subset of 
$M_{\omega+1}$ such that each $tp(c_{i}/M_{\omega})$ is regular. 
By \cite{Pillay-GST}, $M_{\omega + 1}$ is $a$-prime and so also prime over $(M_{\omega},c_{1},\ldots,c_{n})$. 
It remains to be seen that every regular $q\in S(M_{\omega})$ appears among the $tp(c_{i}/M_{\omega})$ up to nonorthogonality. 
But by Lemma 2.5, $p|M_{\omega}$ dominates $q$, so $q$ is realized in $M_{\omega + 1}$ by some $d$, and then $d$ forks with some $c_{i}$ over $M$. 
\end{proof} 

We now aim for a stronger result which decomposes $p|M_{\omega}$  into  a product of weight one types in a stronger sense.  

\begin{Proposition}\label{2.8} 
There are tuples $d_{1},\ldots,d_{n}$ such that:
\begin{itemize}
 \item[(i)] $tp(d_{i}/M_{\omega})$ has weight one and $c_{i}\in acl(M_{\omega},d_{i})$, for each $i$;
 \item[(ii)] $\{d_{1},\ldots,d_{n}\}$ is $M_{\omega}$-independent; and
 \item[(iii)] $e_{\omega}$ is interalgebraic with $(d_{1},\ldots,d_{n})$ over $M_{\omega}$.
\end{itemize}
 \end{Proposition} 

Proposition 2.8 is essentially Lemma 13 of  \cite{Baldwin-Shelah}, although they  have in (iii) only one direction of the interalgebraicity, 
namely that $e_{\omega}$ is algebraic over $M_{\omega},d_{1},\ldots,d_{n}$.  (But the other direction follows automatically as we point out).

So we will give a quick proof of Proposition 2.8  following the same general line of argument as in \cite{Baldwin-Shelah} with a few simplifications.  

\begin{proof}[Proof of Proposition \ref{2.8}] 
\ \\
{\em Claim I.}  There are $d_{1},\ldots,d_{n}$ such that $\{d_{1},\ldots,d_{n}\}$ is $M_{\omega}$-independent, 
and $tp(d_{i},c_{i}/M_{\omega}) = tp(e_{\omega},c_{i}/M_{\omega})$ for each $i$. 
In particular $c_{i}\in acl(M_{\omega},d_{i})$ for each $i$, and $(d_{1},\ldots,d_{n})$ realizes $(p|M_{\omega})^{(n)}$. 
\newline
{\em Proof (of Claim I).}  
Simply choose $d_{i}$ to realize $tp(e_{\omega}/M_{\omega},c_{i})$ such that the $d_{i}$'s are as 
independent as possible over $M_{\omega}$. For example, inductively choose the $d_{i}$ such that $d_{i+1}$ 
is independent from $M_{\omega},d_{1},\ldots,d_{i}$ over $c_{i+1}$. Then the independence of the $c_{i}$ plus forking calculus 
guarantees the independence of the $d_{i}$.

\vspace{2mm}
\noindent
{\em Claim II.} There are $d_{1},\ldots,d_{n}$ as in Claim I, such that $e_{\omega}\in acl(M_{\omega},d_{1},\ldots,d_{n})$.
\newline
{\em Proof (of Claim II).}  Let $M^{n} = acl(M_{\omega},d_{1},\ldots,d_{n})$. 
Then by Lemma 2.2 $M^{n}$ is the prime model over $(M_{\omega},d_{1},\ldots,d_{n})$ so contains a copy of the prime model over $M_{\omega},c_{1},\ldots,c_{n}$. 
So (by Proposition 2.7) we find $e_{\omega}'$ in $M^{n}$ realizing  $tp(e_{\omega}/M_{\omega},c_{1},\ldots,c_{n})$, which suffices. 

\vspace{2mm}
\noindent
Finally we massage the situation in a routine manner to get the full statement of Proposition 2.8. 
For each $i = 1,\ldots,n$ let $f_{i}$ be a tuple such that $c_{i}$ is independent from $c_{i}$ over $M_{\omega}$ and the Morley rank of $tp(d_{i}/M_{\omega},f_{i})$ is minimized. 
Then we know  that $c_{i}$ dominates $d_{i}$ over $M_{\omega}, f_{i}$, whereby $tp(d_{i}/M_{\omega},f_{i})$ has weight one. 
Now choosing the $f_{i}$'s as free as possible, we can ensure that $(f_{1},\ldots,f_{n})$ is independent from 
$(c_{1},\ldots,c_{n})$ over $M_{\omega}$ from which we conclude that $(c_{1},\ldots,c_{n})$ dominates $(d_{1},\ldots,d_{n})$ 
over $(M_{\omega},f_{1},\ldots,f_{n})$.   Let ${\bar f} = (f_{1},\ldots,f_{n})$.  Note that as $(c_{1},\ldots,c_{n})$ 
dominates $e_{\omega}$ over $M_{\omega}$, we have that:
\newline
(*) $e_{\omega}$ is independent from $\bar f$ over $M_{\omega}$. 
\newline
Now choose finite $A\subset M_{\omega}$ such that  $tp(e_{\omega},c_{1},\ldots,c_{n}, d_{1},\ldots,d_{n})/M_{\omega},{\bar f})$  does not fork over $A,{\bar f}$, 
and bearing in mind (*), we may assume that $tp(e_{\omega}/M_{\omega}{\bar f})$ is also definable over $A$. 
Note that we have that $(c_{1},\ldots,c_{n})$ dominates $(d_{1},\ldots,d_{n})$ over $(A,{\bar f})$, and
$e_{\omega}\in acl(d_{1},\ldots,d_{n},A,{\bar f})$.
As $M_{\omega}$ is $\omega$-saturated, we may choose ${\bar f}'$ in $M_{\omega}$ such that $tp({\bar f}/A) = tp({\bar f}'/A)$.  So:
\newline
(a) $tp(a_{\omega},c_{1},\ldots,c_{n},{\bar f}/A) = tp(a_{\omega},c_{1},\ldots,c_{n},{\bar f}'/A)$.
\newline 
So we can choose $(d_{1}',\ldots,d_{n}')$ such that
\newline
(b)  $tp(e_{\omega},c_{1},\ldots,c_{n},d_{1}',\ldots,d_{n}',{\bar f}'/A) = tp(e_{\omega}, c_{1},\ldots,c_{n}, d_{1},\ldots,d_{n},{\bar f}/A)$. 
\newline
In particular:
\newline
(c) $e_{\omega}\in acl(d_{1}',\ldots,d_{n}',A,{\bar f}')$, and 
\newline
(d)  $(c_{1},\ldots,c_{n})$ dominates $(d_{1}',\ldots,d_{n}')$ over $(A,{\bar f}')$ and $tp(d_{i}/A,{\bar f}')$ has weight $1$.
\newline
But $(c_{1},\ldots,c_{n})$ is independent from $M_{\omega}$ over $(A,{\bar f}')$, 
so by (d) we see that each $d_{i}$ is independent from $M_{\omega}$ over $(A,{\bar f}')$ whereby
\newline
(e)  each $tp(d_{i}'/M_{\omega})$ has weight $1$,  $c_{i}$ dominates $d_{i}$ over $M$, $(c_{1},\ldots,c_{n})$ 
dominates $(d_{1}',\ldots,d_{n}')$ over $M_{\omega}$, and $\{d_{1}',\ldots,d_{n}'\}$ is $M_{\omega}$-independent. 
\newline
So renaming $d_{i}'$ as $d_{i}$, we have (i) and (ii) of Proposition 2.8, as well as $e_{\omega}$ 
being algebraic over $(M_{\omega},d_{1},\ldots,d_{n})$. 
To see that the $d_{i}$ are  in $acl(M_{\omega},e_{\omega})$, we do the following. As $(c_{1},\ldots,c_{n})$ dominates $(d_{1},\ldots,d_{n})$ over $M_{\omega}$ 
we can find a copy $M'$ of the $a$-prime (so prime) model over $(M_{\omega},c_{1},\ldots,c_{n})$ which contains $(d_{1},\ldots,d_{n})$. 
By (c), $e_{\omega}\in M'$. 
By Proposition 2.7, $M'$ is also prime over $(M_{\omega},e_{\omega})$ 
so by uniqueness equals $M_{\omega + 1} = acl(M_{\omega},e_{\omega})$ so each $d_{i}\in acl(M_{\omega}, e_{\omega})$.
\end{proof}

We obtain the following ``structure theorem", which is our version  of Theorem 2 from \cite{Baldwin-Shelah}.

\begin{Proposition} 
Let $M$ be a model of $T$ containing $M_{\omega}$. Then there are $J_{1},\ldots, J_{k}$ each being a 
Morley sequence in some weight $1$ type over $M_{\omega}$  such that $M$ is the {\em algebraic closure}  of $M_{\omega}$ union the $J_{i}$. 
\end{Proposition}
\begin{proof} 
For simplicity we assume that in Proposition 2.7, the (strongly) regular types $q_{i} = tp(c_{i}/M_{\omega})$ are pairwise orthogonal. 
In Proposition 2.8, we may assume that $c_{i}$ is a subtuple of $d_{i}$ for $i=1,\ldots,n$. As $c_{i}$ dominates $d_{i}$ over $M_{\omega}$ 
it follows that $tp(d_{i}/M_{\omega},c_{i})$ is actually isolated, by formula $\phi_{i}(y_{i},c_{i})$ say ($\phi(y,z)$ over $M_{\omega}$). 
Let $r_{i} = tp(d_{i}/M_{\omega})$. Now let $M'$ be any model containing $M_{\omega}$. For $i=1,\ldots,n$, let $K_{i}$ be a Morley sequence of $q_{i}$ in $M'$. 
So as in Proposition 1.7, $M'$ is prime over $M_{\omega}\cup\bigcup_{i}K_{i}$.  Now for each $c_{i,j}\in K_{i}$, let $d_{i,j}\in M'$ be such that
$\models \phi_{i}(d_{i,j},c_{i,j})$. So $d_{i,j}$ realizes $r_{i}$ and $\{d_{i,j}:i,j\}$ is $M_{\omega}$-independent. In any case let $J_{i} = (d_{i,j})_{j}$, 
a Morley sequence in the weight $1$-type $r_{i}$, which is contained in $M'$.
\newline
{\em Claim.}  $M' = acl(M_{\omega} \cup\bigcup_{i}J_{i}))$.
\newline
{\em Proof (of claim).}  
In fact it is enough to prove that $acl(M_{\omega}\cup \bigcup_{i}J_{i})$ is a model (elementary substructure of ${\bar M}$), because it will then be prime over 
$(M_{\omega} \cup \bigcup_{i}K_{i})$ so isomorphic to $M'$ (in fact equal to $M'$ as $M'$ is not only prime but also minimal over $M_{\omega}\cup \bigcup_{i}K_{i}$ ). 
Note that in general the $K_{i}$'s may have different cardinalities for different $i=1,\ldots,n$  (in fact some $K_{i}$ may even be empty). Let $J_{i}'$ for $i=1,\ldots,n$ 
be a Morley sequence in $r_{i}$ extending $J_{i}$ such that all the $J_{i}'$ have the same cardinality $\kappa$ say. 
For each $\alpha < \kappa$, let $a_{\alpha}$ be a realization of $p|M_{\omega}$ interalgebraic with $(d_{i,_\alpha}):i=1,\ldots,n)$ 
(where $J_{i} = (d_{i,\alpha}: \alpha < \kappa)$). Then $(a_{\alpha} \alpha < \kappa)$ is a Morley sequence in $p|M_{\omega}$ 
so by Lemma 2.2  its algebraic closure over $M_{\omega}$ is a model.  But this coincides with  $acl(M_{\omega},\bigcup_{i}J_{i}')$ which is therefore a model.  
Now as $\bigcup_{i}J_{i}'$ is independent over $M_{\omega}$, for each tuple $b$ from $\bigcup_{i}J_{i}'\setminus \bigcup_{i}J_{i}$, $tp(b/M_{\omega}\cup\bigcup_{i}J_{i})$ 
is finitely satisfiable in $M_{\omega}$. So using Tarski-Vaught it follows that $acl(M_{\omega}\cup\bigcup_{i}J_{i})$ is an elementary substructure of ${\bar M}$, as required. 
\end{proof}

\begin{Remark} 
In Section 4 we give a few examples of almost indiscernible theories of infinite rank. 
But one can check by inspection that any almost indiscernible theory of abelian groups (in the group language) has finite Morley rank. 
\end{Remark} 


\section{Free Algebras}

The reader is referred to \cite{BS} for background on universal algebra (of which not much is needed). 
As mentioned before we work with algebras in a countable language (or signature) $L$. Fix a variety $V$. 
Then for any set $X$, $F_{X}$ denotes the free algebra in $V$ on generators $X$, and we call $X$ a {\em basis} of $F_{X}$. 
In general it is possible that $F_{X}$ and $F_{Y}$ are isomorphic even though $X$ and $Y$  have different cardinalities, 
so there is no well-defined notion of dimension for a free algebra.  But this can only happen if both $X,Y$ are finite. 
On the other hand it is clear that any bijection between $X$ and $Y$ extends to an isomorphism between $F_{X}$ and $F_{Y}$ 
and conversely any isomorphism between $F_{X}$ and $F_{Y}$ takes $X$ to another basis  of $F_{Y}$. 

In general if $A$ is an algebra and $X$ a subset of $A$ then $\langle X \rangle$ denotes the subalgebra of $A$ generated by $X$.

\begin{Remark} Suppose that the algebra $A$ is free on $X_{1}\cup X_{2}$, and $A_{1}$ is the subalgebra of $A$ (freely) generated by $X_{1}$. 
Let $Y_{1}$ be another basis of $A_{1}$. Then $A$ is freely generated by $Y_{1}\cup X_{2}$.
\end{Remark} 
\begin{proof} 
Let $B$ be an algebra in $V$ and $f:Y_{1}\cup X_{2} \to B$. 
So $f|Y_{1}$ extends uniquely to a homomorphism $f_{1}:A_{1}\to B$. Let $g$ be the restriction of $f_{1}$ to $X_{1}$. 
Then as $A_{1}$ is free on $X{1}$, $f_{1}$ is also the unique extension of $g$ to a homomorphism from $A_{1}$ to $B$. 
Now $g \cup f|X_{2}$ is a map from $X_{1}\cup X_{2}$ to $B$ hence extends to a homomorphism $h$ from $A$ to $B$. 
Now the restriction of $h$ to $A_{1}$ must coincide with $f_{1}$ hence the restriction of $h$ to $Y_{1}$ coincides with $f|Y_{1}$. So $h$ extends $f$. 
\end{proof}

\noindent
{\bf Assumption.} ${\cal A} = M$ is a free algebra for $V$ on a set  $I = (e_{\alpha}:\alpha< \aleph_{1})$ and is moreover $\aleph_{1}$-saturated. 

\vspace{2mm}
\noindent
So $I$ is an uncountable indiscernible set in $M$, $dcl(I) = M$ and $M$ is saturated, whereby all of section 2 applies to $T = Th(M)$.  But we will prove some results which are specific to the ``free saturated algebra" setting.

It is also not hard to see that if $I'$ is either a shrinking or stretching of $I$ to another infinite indiscernible set in the sense of $T$, 
then the algebraic closure of $I'$ (in the ambient model of $T$) is precisely $\langle I' \rangle$ and is moreover free on $I'$ in the variety $V$. 

\begin{Definition} 
We call a subset $A$ of $M$ {\em basic} if $A$ is a subset of a basis of $M$. And we call an element $a\in M$ basic if $\{a\}$ is basic.  
\end{Definition} 

So a basic element is what in the context of a free group is called a primitive element.

\begin{Lemma} 
There is a complete type $p_{0}(x)$ over $\emptyset$ such that for any $a\in M$, $a$ is basic if and only if $a$ realizes $p_{0}$. 
\end{Lemma} 
\begin{proof} Note that all  elements of $I$ have the same type over $\emptyset$ which we take to be $p_{0}(x)$. 
Suppose first that $a$ is basic. So $a$ extends to a basis $X$ for $M$. $X$ has cardinality $\aleph_{1}$ too 
and any bijection between $X$ and $I$ induces an automorphism of $M$, so $a$ realizes $p_{0}$. 
Conversely if $a$ realizes $p_{0}$ in $M$ and $e\in I$ then there is an automorphism of $M$ taking $e$ to $a$ 
(as $M$ is saturated so homogeneous) and the image of $I$ will be a basis of $M$ containing $a$. 
\end{proof}

\begin{Remark}
As remarked above, if  $I_{0}$ is a countable subset of $I$ and $M_{0}$ is the subalgebra of $M$ generated by $I_{0}$, 
then $M_{0}$ is free on basis $I_{0}$, and is moreover an $\omega$-saturated elementary substructure of $M$. 
In particular $p_{0}$ is the type of any element of $I_{0}$ in $M_{0}$, and Lemma 3.3 also applies to $M_{0}$ with the same proof. 
\end{Remark}

\begin{Question} 
Is $p_{0}(x)$  of maximal Morley rank among complete $1$-types of $T$?
\end{Question}

\begin{Lemma} 
Suppose $a\in M$ is basic and $a$ is a term in $e_{\alpha_{1}},\ldots,e_{\alpha_{n}}$, 
then for any countable subset $C$ of  $I\setminus\{e_{\alpha_{1}},\ldots,e_{\alpha_{n}}\}$, $C\cup\{a\}$ is a basic set. 
\end{Lemma}
\begin{proof}  
Extend $e_{\alpha_{1}},\ldots,e_{\alpha_{n}}$ to a countable subset $I_{0}$ of $I$, avoiding $C$. 
Let $M_{0}$ be generated by $I_{0}$. Then $a\in M_{0}$ and by Remark 3.4 is basic in $M_{0}$. 
By Remark 3.1 $\{a\}\cup C$ is basic in $M$, and also basic in the (free) algebra it generates. 
\end{proof} 

\begin{Lemma} 
$p_{0}$ is stationary (as therefore is $p_{0}^{(n)}$ for any $n$). 
\end{Lemma}
\begin{proof} 
We have to show that $p_{0}$ determines a unique strong type over $\emptyset$. 
So suppose $a,b$ are both realizations of $p_{0}$. So $a$ is part of a  basis $I$ of $M$ and $b$ part of a basis  $J$ of $M$. 
By Lemma 3.6  there is $b'\in J$ such that $\{a,b'\}$ is basic, namely extends to another  basis $J'$ of $M$. 
But then, as $J'$ is indiscernible in $M$, $a$ and $b'$ have the same strong type. 
As for the same reason $b$ and $b'$ have the same strong type it follows that $a$ and $b$ do too.
\end{proof} 

\begin{Proposition} 
$I$ is a Morley sequence in $p_{0}$, namely not only indiscernible but also independent over $\emptyset$. 
\end{Proposition}

\begin{proof} 
Let $I_{0} = \{e_{\alpha}:\alpha < \omega\}$. 
Let $a$ realizes $p_{0}$ such that $a$ is independent from $I_{0}$ over $\emptyset$. 
By Lemma 3.6 we can find an infinite subset $I_{0}'$ of $I_{0}$ such that $I_{0}'\cup\{a\}$ is a 
basic. But then this is an indiscernible set with the same ``Ehrenfeucht-Mostoswki"  type as $I$. 
Hence  example $a_{\omega}$ is independent from $I_{0}$ over $\emptyset$ which is enough.
\end{proof} 

\begin{Corollary} 
In $T$, $acl^{eq}(\emptyset) = dcl^{eq}(\emptyset)$. 
\end{Corollary}
\begin{proof} 
Suppose $a\in acl^{eq}(\emptyset)$. Then $a\in dcl^{eq}({\bar e})$ for some finite tuple from $I$. 
But by Lemma 3.7 and Proposition 3.8, $tp({\bar e}/\emptyset)$ is stationary, whereby $tp(a/\emptyset)$ is stationary whereby $a\in dcl^{eq}(\emptyset)$. 
\end{proof}

We can prove in a similar manner.

\begin{Remark} 
Let $E$ be any subset of $I$ (or in fact any basic set). Then $acl^{eq}(E) =dcl^{eq}(E)$. 
\end{Remark}

\begin{Proposition} 
If $\bar a$, $\bar b$ are tuples from $M$. Then $\bar a$ is independent from $\bar b$ over $\emptyset$ if and only if 
there is a basis $B_{1}\cup B_{2}$ of $M$ such that ${\bar a}$ is contained in $\langle B_{1}\rangle$ and $\bar b$ is contained in  $\langle B_{2}\rangle$.
\end{Proposition}
\begin{proof} 
Right implies left is clear as by a basis (or basic subset) of $M$ is independent over $\emptyset$.
For the converse. Suppose $\bar a$ and $\bar b$ are independent over $\emptyset$. 
Without loss, $\bar a$, $\bar b$  are both terms in $e_{1},\ldots,e_{n}$ and write $\bar a = \bar t(e_{1},\ldots,e_{n})$. Let ${\bar a}' {\bar t}(e_{n+1},\ldots,e_{2n})$. 
Then ${\bar a}'$ is independent from ${\bar b}$ and $tp(\bar a) = tp({\bar a}')$. By stationarity $tp({\bar a},{\bar b}) = tp({\bar a}',{\bar b})$, 
so by automorphism we can find the suitable basis. 
\end{proof}

\begin{Remark}
Proposition 3.11  extends naturally to describing independence {\em over} any  basic set $B$.
\end{Remark}

The interested reader can consult \cite{PerSkl} for the analogous result for noncommutative free groups. 
As a matter of fact our proof is a straight adaptation of the proof there.

\begin{Question}  Let  $T$ be the theory of  saturated free algebra. 
\begin{itemize}
\item[(i)]Must $T$ have finite Morley rank?
\item[(ii)] Must $T$ be $1$-based.
\end{itemize}
\end{Question}

Given (i), one could prove (ii) by showing that inside suitable strongly minimal sets, 
algebraic closure equals definable closure, so we have "unimodularity" so one-basedness.
Probably (i) is easy for saturated free algebras in a variety of $R$-modules.

One could also specialize to the context where $V$ is a variety of groups (in the language of groups including an inverse function).

\begin{Question} Suppose $V$ is a variety of groups and $G$ is free in $V$ as well as being saturated.
Is $G$ commutative and of finite Morley rank (in which one can explicitly list the possibilities).
\end{Question}

\begin{Remark} 
In the context where $T$ is the theory of a saturated free group $G$  in a variety of groups, 
we have that  $G$ is connected and $p_{0}$ is the generic type over $\emptyset$. 
This follows as in the case of the free group (see \cite{Pillay-freegroup}). 
\end{Remark}

\section{Examples}

We typically work with one-sorted structures, where the relevant indiscernible set witnessing slmost indiscernibility is a set of $n$-tuples for some $n$ (rather than working in a many sorted theory $T^{eq}$). 

\begin{Example} Any almost strongly minimal theory is almost indiscernible ( after adding additional parameters to witness the almost strong minimality).
\end{Example}

The next two examples give almost indiscernible theories of infinite rank (which we conjectured could not happen for the theory of a saturated free algebra)

\begin{Example} 
Consider the theory $T$ of infinitely many disjoint infinite unary predicates $P_{1}, P_{2},\ldots$ equipped with, for each $n$, 
a bijection $f_{n}$ between $P_{1}^{n}$ and $P_{n}$.  The theory is complete. $P_{n}$ has Morley rank $n$, 
whereby the Morley rank of the universe $x=x$ is $\omega$. $P_{1}$ is an indiscernible set. Let $q(x)$ be the ``type at infinity":  $\{\neg P_{n}(x): n=1,2,\ldots\}$. 
Then $q$ is complete with $U$-rank $1$ (and Morley rank $\omega$), and its set of realizations in any model is also an indiscernible set. 
Let $M$ be a $\kappa$-saturated model of $T$ of cardinality $\kappa$. Let $(a_{i}:i<\kappa)$ be an enumeration of $P_{1}$ in $M$, 
and let $(b_{i}:i<\kappa)$ be an enumeration of the srt of realizations of $q$ in $M$.
Then $((a_{i}b_{i}):i<\kappa)$ is an indiscernible set in $M$ whose definable closure is precisely $M$. 
This is the simplest example of an almost indiscernible theory with Morley rank of $x=x$ being infinite. 
\end{Example} 

\begin{Example} 
This is a kind of group version of the above. Let $T$ be the theory of $\Q$-vector spaces equipped with a new predicate $P$ for an infinite $\Q$-linearly independent set. 
Then $T$ is complete, and $\omega$-stable. The predicate $P$ is strongly minimal (and its solution set in any model is an indiscernible set). $nP = P+ \ldots + P$ (n-times) 
has Morley rank $n$. The formula $x=x$ has Morley rank $\omega$ again, but also $U$-rank $\omega$: 
Let $q$ be the type saying $\{x\notin nP: n=1,2,\ldots\}$ . Then $q$ has $U$-rank and Morley rank $\omega$. 
Let $M$ be a $\kappa$-saturated model of $T$. 
Let $P(M) = (a_{i}:i<\kappa)$ and let $(b_{i}:i<\kappa)$ be a maximal independent set of realizations of $q$ in $M$. 
Then $M$ is in the definable closure of $((a_{i}b_{i}):i:\kappa)$. 

The same thing can be done with the theory of algebraically closed fields in place of $\Q$-vector spaces.
\end{Example}

\begin{Example} This is actually a nonexample.  Let $T$ be theory  of $({\Z}_{p^{\infty}}^{\omega}, +)$, 
where ${\Z}_{p^{\infty}}$ is the group of roots of unity of order a power of $p$.  So $T$ is a theory of 
abelian groups of Morley rank $\omega$. 
A  $\kappa$-saturated model $M$ of $T$ is of the form $({\Z}_{p^{\infty}})^{\kappa} \oplus {\Q}^{\kappa}$. 
$M$ is {\em not} in the algebraic closure of an indiscernible set of finite tuples of cardinality $\kappa$. 
But it {\em is} visibly in the definable closure of an indiscernible set of $\omega$-tuples of cardinality $\kappa$.

\end{Example}

It is easy to produce an almost indiscernible theory, witnessed by an infinite indiscernible set $I$ such that $tp(a/\emptyset)$, for $a\in I$ is not stationary. 
For example the theory of an equivalence relation with two classes, both infinite. But we would like an example 
of an almost indiscernible $\aleph_{1}$-categorical theory, such that  there is no Morley sequence (so infinite, indiscernible, and independent) witnessing almost indiscernibility.

\end{document}